\documentclass{amsart}
\usepackage[cmtip,all]{xy}
\newdir{ >}{{}*!/-5pt/@{>}} 
\usepackage{amssymb} 
\usepackage[alphabetic,nobysame]{amsrefs}

\title{The Segal conjecture for smash powers}
\author{H{\aa}kon Schad Bergsaker and John Rognes}
\address{Department of Mathematics, University of Oslo, Norway}
\email{hsbergsaker@gmail.com}
\address{Department of Mathematics, University of Oslo, Norway}
\email{rognes@math.uio.no}

\dedicatory{Dedicated to our PhD advisor and grand-advisor
	Gunnar Carlsson, \\ on the occasion of his 70th birthday}

\date{January 12th 2023}
\thanks{Part of this research was done while the first author was
	supported by the ``Topology in Norway'' project FRINATEK
	ES479962.}

\newtheorem{theorem}{Theorem}[section]
\newtheorem{addendum}[theorem]{Addendum}
\newtheorem{proposition}[theorem]{Proposition}
\newtheorem{lemma}[theorem]{Lemma}

\theoremstyle{definition}
\newtheorem{definition}[theorem]{Definition}

\theoremstyle{remark}
\newtheorem{remark}[theorem]{Remark}

\numberwithin{equation}{section}

\DeclareMathOperator{\Aut}{Aut}
\DeclareMathOperator*{\colim}{colim}
\DeclareMathOperator{\End}{End}
\DeclareMathOperator{\Ext}{Ext}
\DeclareMathOperator*{\holim}{holim}
\DeclareMathOperator{\Hom}{Hom}
\DeclareMathOperator{\res}{res}
\DeclareMathOperator{\St}{St}
\DeclareMathOperator{\Tor}{Tor}
\newcommand{\bC}{\mathbb{C}}
\newcommand{\bF}{\mathbb{F}}
\newcommand{\bZ}{\mathbb{Z}}
\newcommand{\cB}{\mathcal{B}}
\newcommand{\cP}{\mathcal{P}}
\newcommand{\cT}{\mathcal{T}}
\newcommand{\longfrom}{\longleftarrow}
\newcommand{\longto}{\longrightarrow}
\newcommand{\wEG}{\widetilde{EG}}
\newcommand{\wEK}{\widetilde{EK}}
\newcommand{\wEP}{\widetilde{E{\cP}}}
\renewcommand{\:}{\colon}

\begin{document}

\begin{abstract}
We prove that the comparison map from $G$-fixed points to $G$-homotopy
fixed points, for the $G$-fold smash power of a bounded below
spectrum~$B$, becomes an equivalence after $p$-completion if $G$ is a
finite $p$-group and $H_*(B; \bF_p)$ is of finite type.  We also prove
that the map becomes an equivalence after $I(G)$-completion if $G$
is any finite group and $\pi_*(B)$ is of finite type, where $I(G)$
is the augmentation ideal in the Burnside ring.
\end{abstract}


\maketitle

\section{Introduction}

Let $G$ be a finite group, let $B$ be a flat orthogonal spectrum, and let
$$
B^{\wedge G} = \bigwedge_{g \in G} B
$$
be the $G$-fold smash power of~$B$, i.e., the smash product of
one copy of $B$ for each element of~$G$.  The group~$G$ acts from
the left on $B^{\wedge G}$ by permuting the smash factors, and the
resulting orthogonal spectrum with $G$-action prolongs essentially
uniquely to an orthogonal $G$-spectrum indexed on any given choice of
a complete $G$-universe.  This construction is originally due to Marcel
B{\"o}kstedt (ca.~1987, cf.~\cite{HM97}*{\S2.4}), who worked in the
context of functors with smash product.  In the context of orthogonal
spectra it is the special case $B^{\wedge G} = N_{\{e\}}^G B$ of the
Hill--Hopkins--Ravenel~\cite{HHR16} norm.  When $B = S$ is the sphere
spectrum, this construction produces the $G$-equivariant sphere spectrum
$S^{\wedge G} = S_G$.


For any $G$-spectrum $X$ there is a comparison map
$$
\gamma \: X^G = F(S^0, X)^G \longto F(EG_+, X)^G = X^{hG}
$$
from $G$-fixed points to $G$-homotopy fixed points, induced by the
collapse map $c \: EG_+ \to S^0$.  Let $p$ be a prime and suppose for
a little while that $G$ is a $p$-group.  We then say that the (generalized)
Segal conjecture holds for the $G$-spectrum $X$ if the comparison
map $\gamma$ becomes an equivalence after $p$-completion.  When $X =
S_G$ this is equivalent to Graeme Segal's Burnside ring conjecture
for the $p$-group~$G$, in the strong form proved by Gunnar Carlsson.
We adapt the overall strategy \cite{Car84}, \cite{AGM85}, \cite{PW85},
\cite{CMP87} from Carlsson's proof to establish the following result,
which specializes to his theorem in the case $B = S$.

\begin{theorem} \label{thm:segal-smash-pgroup}
Let $p$ be a prime, $G$ a finite $p$-group and $B$ a flat orthogonal
spectrum.  Suppose that $\pi_*(B)$ is bounded below and that $H_*(B;
\bF_p)$ is of finite type.  Then the Segal conjecture holds for the smash
power $G$-spectrum $B^{\wedge G}$.  In other words, the comparison map
$$
\gamma \: (B^{\wedge G})^G \longto (B^{\wedge G})^{hG}
$$
becomes an equivalence after $p$-completion.
\end{theorem}

The proof is given near the end of Section~\ref{sec:towers} for $G$ not
elementary abelian, and at the end of Section~\ref{sec:elemabel} for $G
\cong (C_p)^k$ with $k\ge1$.  The main novelty of our work concerns how
we deal with the fact that in general the $G$-spectrum $B^{\wedge G}$
is not split in the sense of May--McClure~\cite{MM82}*{Def.~10}, so that
\cite{CMP87}*{Thm.~B} does not apply, even though this is so for $B = S$.

When $G \cong C_p$ the theorem was proved earlier by Sverre
Lun{\o}e--Nielsen and the second author~\cite{LNR12}*{Thm.~5.13},
and the case $G \cong C_{p^n}$ was proved by these authors together
with Marcel B{\"o}kstedt and Robert Bruner in~\cite{BBLNR14}*{Thm.~2.7}.
For $G \cong C_p$ the finite type hypothesis on $H_*(B; \bF_p)$ was
subsequently lifted by Nikolaus--Scholze in~\cite{NS18}*{Thm.~III.1.7}.
We do not know whether the finite type hypothesis can be removed for $G$
containing elementary abelian $p$-groups $(C_p)^k$ of rank $k\ge2$.

Now return to the case of a general finite group~$G$.
Let $I(G) \subset A(G)$ denote the augmentation ideal in the Burnside
ring.  For any $G$-spectrum~$X$ the comparison map
$$
\gamma \: X^G \overset{\iota}\longto (X^\wedge_{I(G)})^G
	\overset{\xi^*}\longto X^{hG}
$$
extends naturally over the spectrum level $I(G)$-completion map here
denoted~$\iota$, cf.~Greenlees--May~\cite{GM92}*{\S4}.  We now say that
the (generalized) Segal conjecture holds for the $G$-spectrum~$X$ when
the natural extension~$\xi^*$ is an equivalence.  When $I(G)$-completion
induces $I(G)$-adic completion at the level of $G$-equivariant homotopy
groups, May--McClure~\cite{MM82}*{p.~217} refer to this assertion about
$\xi^*$ as the completion conjecture, and such results are referred
to as completion theorems in~\cite{GM92}.  In particular, the Segal
(or completion) conjecture for the $G$-spectrum~$S_G$ is equivalent
to the strong form of Segal's Burnside ring conjecture for the general
finite group~$G$.

When $G$ is a $p$-group, it follows from work of K{\'a}ri
Ragnarsson~\cite{Rag11}*{Thm.~C}, adapting~\cite{MM82}*{Prop.~14} to
the non-split case, that the two formulations just given of the Segal
conjecture agree for bounded below $G$-spectra~$X$ with $\pi_*(X)$ of
finite type, since the comparison map~$\gamma$ becomes an equivalence
after $p$-completion if and only if it becomes one after $I(G)$-adic
completion.  See Proposition~\ref{prop:peqIKeq}.  Hence we can apply
\cite{MM82}*{Thm.~13} to deduce the following form of the Segal conjecture
for general finite groups and their smash power spectra.

\begin{theorem} \label{thm:segal-smash-finite}
Let $G$ be a finite group and $B$ a flat orthogonal spectrum.  Suppose
that $\pi_*(B)$ is bounded below and of finite type.  Then the Segal
conjecture holds for the smash power $G$-spectrum $B^{\wedge G}$.
In other words, the natural map
$$
\xi^* \: ((B^{\wedge G})^\wedge_{I(G)})^G
	\longto (B^{\wedge G})^{hG}
$$
is an equivalence, inducing an isomorphism
$$
\pi_*(\gamma)^\wedge_{I(G)} \: \pi_*((B^{\wedge G})^G)^\wedge_{I(G)}
	\overset{\cong}\longto \pi_*((B^{\wedge G})^{hG})
$$
of $I(G)$-adically completed homotopy groups.
\end{theorem}

We give the proof at the end of Section~\ref{sec:finite}.  We make
the assumption that $\pi_*(B)$ is of finite type in order to ensure
that the spectrum level $I(G)$-completion induces algebraic $I(G)$-adic
completion at the level of $G$-equivariant homotopy groups, so as to be
able to refer directly to the algebraic induction theory of~\cite{MM82}.
Presumably this can be sidestepped by carrying out the induction theory
closer to the spectrum level.


The first author obtained a proof of Theorem~\ref{thm:segal-smash-pgroup}
around April 2013, and lectured on the result at a conference in December
2015~\cite{Ber15}.  The second author returned to the argument in May and
June 2022, finding some simplifications that have been incorporated into the
present account.  We apologize for the long delay in publication.
In his July 2022 ICM address~\cite{Nik22}*{Rem.~7.11},
Thomas Nikolaus conjectured that Theorem~\ref{thm:segal-smash-pgroup}
also holds without the finite type assumption on mod~$p$ homology.
As mentioned above, we do not know how to remove this hypothesis.

The hallmark signs of Gunnar Carlsson's breakthrough approach to the
classical Segal conjecture are evident throughout our paper.  We heartily
congratulate him on the occasion of his anniversary.

\section{Isotropy separation and $S$-functors}

To prove the Segal conjecture for the $G$-spectra $B^{\wedge G}$, we
follow Carlsson and inductively assume that it holds for the $J$-spectra
$C^{\wedge J}$ for all proper subquotient groups $J = K/H$ of $G$.
This is useful, because of the following proposition.

\begin{proposition} \label{prop:resgeomfix}
Let $H \lhd K \subset G$, let $J = K/H$, and let $B$ be a flat orthogonal
spectrum.

(a)
The restriction $\res^G_K(B^{\wedge G})$ along $K \subset G$ of the
$G$-spectrum $B^{\wedge G}$ is equivalent to the $K$-spectrum $C^{\wedge
K}$, where
$$
C = B^{\wedge G/K} = \bigwedge_{Kg \in G/K} B
$$
is the smash product of one copy of $B$ for each right coset of $K$
in~$G$.

(b)
The geometric $H$-fixed point spectrum $\Phi^H(C^{\wedge K})$ of the
$K$-spectrum $C^{\wedge K}$ is equivalent to the $J$-spectrum $C^{\wedge
J}$.

(c)
If $\pi_*(B)$ is bounded below and $H_*(B; \bF_p)$ is of finite type,
then $\pi_*(C)$ is bounded below and $H_*(C; \bF_p)$ is of finite type.

(d)
If $\pi_*(B)$ is bounded below and of finite type, then $\pi_*(C)$
is bounded below and of finite type.
\end{proposition}

\begin{proof}
See~\cite{HM97}*{Prop.~2.5} or~\cite{HHR16}*{Prop.~B.209} for part~(b).
The remaining claims are clear.
\end{proof}

As usual, let $EG$ denote any $G$-CW space with $EG^{\{e\}}$ contractible
and $EG^K$ empty for each nontrivial subgroup $K \subset G$.
Also let $E\cP$ denote any $G$-CW space with $E\cP^K$ contractible
for each proper subgroup $K \subset G$ and $E\cP^G$ empty.
Define $\wEG$ and $\wEP$ by the homotopy cofiber sequences
\begin{align*}
EG_+ &\overset{c}\longto S^0 \longto \wEG \\
E\cP_+ &\overset{c}\longto S^0 \longto \wEP \,,
\end{align*}
where the collapse maps~$c$ send $EG$ and $E\cP$ to the non-base point
of $S^0$.
Following~\cite{Car84}*{\S III}, let $\rho = \ker(\epsilon \: \bC\{G\}
\to \bC)$ be the reduced regular complex representation of~$G$.
Then $S^{\infty\rho}$ is a model for $\wEP$, conveniently filtered by
the $G$-CW subspaces
\begin{equation} \label{eq:Smrhofiltration}
S^0 \subset \dots \subset S^{m\rho} \subset \dots \subset S^{\infty\rho}
\,.
\end{equation}

\begin{proposition}[\cite{Car84}*{Thm.~A(b)}, \cite{CMP87}*{Lem.~1.9}]
\label{prop:SinftyrhoBGG}
Let $G$ be a nontrivial $p$-group and suppose that
Theorem~\ref{thm:segal-smash-pgroup} holds for each proper subgroup
of~$G$.  Let $B$ be a flat orthogonal spectrum with $\pi_*(B)$ bounded
below and $H_*(B; \bF_p)$ of finite type, and suppose also that
$$
F(S^{\infty\rho}, B^{\wedge G})^G
$$
becomes trivial after $p$-completion.  Then
$$
\gamma \: (B^{\wedge G})^G \longto (B^{\wedge G})^{hG}
$$
becomes an equivalence after $p$-completion.
\end{proposition}

\begin{proof}
Consider the commutative square
$$
\xymatrix{
F(S^0, B^{\wedge G})^G \ar[r]^-{\gamma} \ar[d]
	& F(EG_+, B^{\wedge G})^G \ar[d] \\
F(E\cP_+, B^{\wedge G})^G \ar[r]
	& F(EG_+ \wedge E\cP_+, B^{\wedge G})^G \,.
}
$$
The right hand arrow is an equivalence because
$EG \times E\cP \simeq_G EG$.
The lower arrow is a homotopy limit of maps
$$
(B^{\wedge G})^K \cong F(G/K_+, B^{\wedge G})^G
	\longto F(EG_+ \wedge G/K_+, B^{\wedge G})^G
	\simeq (B^{\wedge G})^{hK} \,,
$$
with $K$ ranging over the proper subgroups of $G$, and therefore becomes
an equivalence after $p$-completion by the inductive hypothesis
and Proposition~\ref{prop:resgeomfix}(a,c).  The left hand arrow becomes an
equivalence after $p$-completion if and only if its homotopy fiber, namely
$F(S^{\infty\rho}, B^{\wedge G})^G$, becomes trivial after $p$-completion.
\end{proof}

We continue to follow Carlsson's strategy of isotropy separation,
considering the homotopy cofiber sequence
\begin{equation} \label{eq:deltaseq}
F(S^{\infty\rho}, \Sigma^{-1} \wEG \wedge B^{\wedge G})^G
	\overset{\delta}\longto
F(S^{\infty\rho}, EG_+ \wedge B^{\wedge G})^G \\
	\overset{c}\longto
F(S^{\infty\rho}, B^{\wedge G})^G \,.
\end{equation}
Clearly $F(S^{\infty\rho}, B^{\wedge G})^G$ becomes trivial after
$p$-completion if and only if the connecting map $\delta$ becomes an
equivalence after $p$-completion, and this is what we will verify.
We note that
$$
F(S^{\infty\rho}, \Sigma^{-1} \wEG \wedge B^{\wedge G})^G
\simeq \holim_m F(S^{m\rho}, \Sigma^{-1} \wEG \wedge B^{\wedge G})^G
$$
and
\begin{align*}
F(S^{\infty\rho}, EG_+ \wedge B^{\wedge G})^G
&\simeq \holim_m F(S^{m\rho}, EG_+ \wedge B^{\wedge G})^G \\
&\simeq \holim_m \, (\Sigma^{2m} EG_+ \wedge (\Sigma^{-2m} B)^{\wedge G})^G \\
&\simeq \holim_m \Sigma^{2m} EG_+ \wedge_G (\Sigma^{-2m} B)^{\wedge G} \\
&= \holim_m \Sigma^{2m} D_G (\Sigma^{-2m} B) \,.
\end{align*}
Here the first two equivalences are induced by the
filtration~\eqref{eq:Smrhofiltration}, the third equivalence follows
from an identification $S^{m\rho} \wedge S^{2m} \cong (S^{2m})^{\wedge
G}$, and the fourth equivalence is a case of the Adams transfer
equivalence~\cite{LMS86}*{Thm.~II.7.1}.  The final identity uses the
notation
$$
D_G B = EG_+ \wedge_G B^{\wedge G}
$$
for the $G$-fold extended power of any spectrum~$B$, where $G$ is viewed
as a subgroup of the symmetric group on $|G|$ elements.  In particular,
$D_G S = BG_+$.  The map in the limit system that corresponds to
restriction along $S^{m\rho} \subset S^{(m+1)\rho}$ is then the twisted
diagonal map
\begin{multline*}
\Sigma^{2(m+1)} D_G(\Sigma^{-2(m+1)} B)
	= \Sigma^{2m} \Sigma^2 D_G(\Sigma^{-2(m+1)} B) \\
	\overset{\Delta}\longto \Sigma^{2m} D_G(\Sigma^2 \Sigma^{-2(m+1)} B)
	\simeq \Sigma^{2m} D_G(\Sigma^{-2m} B)
\end{multline*}
of~\cite{BMMS86}*{Def.~II.3.1}, associated to the based CW space
$S^2$.  For brevity we introduce the following notations.

\begin{definition} \label{def:VWdelta}
Let
$$
V(G, B) = \holim_m F(S^{m\rho}, \Sigma^{-1} \wEG \wedge B^{\wedge G})^G
$$
and
$$
W(G, B) = \holim_m \Sigma^{2m} D_G (\Sigma^{-2m} B)
$$
define functors of~$B$,
so that there is a natural homotopy cofiber sequence
$$
V(G, B) \overset{\delta}\longto W(G, B)
	\longto F(S^{\infty\rho}, B^{\wedge G})^G \,.
$$
\end{definition}

For $G$-spectra~$X$, the spectra $F(S^{\infty\rho}, \wEG \wedge X)^G$,
hence also the spectra~$V(G, B)$, have been fully analyzed by means of
Carlsson's theory of $S$-functors~\cite{Car84}*{\S IV--VI}.  Recall that
an elementary abelian $p$-group is a group of the form $G \cong (C_p)^k$.
The rank~$k\ge1$ Tits building~$\cT_k$ is the classifying space of the
partially ordered set of proper, nontrivial subgroups of $(C_p)^k$,
and by the Solomon--Tits theorem~\cite{Sol69} its double suspension
$$
\Sigma^2 \cT_k \simeq \bigvee^{p^{\binom{k}{2}}} S^k
$$
has the homotopy type of a finite wedge sum of $k$-spheres.  Here
$p^{\binom{k}{2}}$ denotes $p$ raised to the power ${\binom{k}{2}}
= k(k-1)/2$.  The wedge sum in the following result suggested the use
of the letter `$V$' in $V(G, B)$.

\begin{theorem}[\cite{Car84}*{\S\S IV--VI}, \cite{CMP87}*{\S\S 3--4}]
\label{thm:VGB}
Let $G$ be a nontrivial $p$-group and suppose that
Theorem~\ref{thm:segal-smash-pgroup} holds for each proper subquotient
of~$G$.  Let $B$ be a flat orthogonal spectrum with $\pi_*(B)$ bounded
below and $H_*(B; \bF_p)$ of finite type.

(a)
If $G = (C_p)^k$, then there are natural equivalences
$$
V(G, B)^\wedge_p \simeq F(\Sigma^2 \cT_k, B)^\wedge_p
	\simeq \bigvee^{p^{\binom{k}{2}}} \Sigma^{-k} B^\wedge_p
\,.
$$

(b)
If $G$ is not elementary abelian, then $V(G, B)^\wedge_p \simeq *$.
\end{theorem}

\begin{proof}
This is the special case $k_G = B^{\wedge G}$, $j = k_{G/G} =
\Phi^G(B^{\wedge G})$ of Caruso--May--Priddy's~\cite{CMP87}*{Thm.~A},
in view of the equivalence~$\Phi^G(B^{\wedge G}) \simeq B$ recalled in
Proposition~\ref{prop:resgeomfix}(b).
\end{proof}

As pointed out in~\cite{CMP87}*{Rem.~8.4}, for $G = (C_p)^k$
there is a natural action of $GL_k(\bZ/p)$ on the terms in the
sequence~\eqref{eq:deltaseq}, and the maps are $GL_k(\bZ/p)$-equivariant.
This uses that the $G$-actions on $B^{\wedge G}$, $EG_+ \to S^0 \to \wEG$
and $S^{\infty\rho}$ all extend to permutation actions by the symmetric
group~$\Sigma_{|G|}$.  Hence the normalizer~$N$ of $G$ in~$\Sigma_{|G|}$
acts naturally on the $G$-fixed point spectra in~\eqref{eq:deltaseq},
and these actions factor through the Weyl group $N/G$.  This normalizer
is classically known as the holomorph of $G$, and is isomorphic to the
semidirect product $\Aut(G) \ltimes G$ for the tautological action of
the automorphism group~$\Aut(G)$ on $G$.  In the case $G = (C_p)^k$,
the normalizer is the semidirect product $N \cong GL_k(\bZ/p) \ltimes
(C_p)^k$.  Here the Weyl group $N/G = GL_k(\bZ/p)$ acts linearly
on $(C_p)^k$, via $\bZ/p = \End(C_p)$.

Similarly, $GL_k(\bZ/p)$ acts on the partially ordered set of proper,
nontrivial subgroups of $(C_p)^k$, hence also on $\cT_k$ and $F(\Sigma^2
\cT_k, B)^\wedge_p$, and the first equivalence in Theorem~\ref{thm:VGB}(a)
respects these $GL_k(\bZ/p)$-actions.  The induced action on
$$
\St_k := H_k(\Sigma^2 \cT_k; \bZ) \cong \bigoplus^{p^{\binom{k}{2}}} \bZ
$$
is the Steinberg representation.  Let $U_k(\bZ/p) \subset GL_k(\bZ/p)$
be the subgroup of upper triangular matrices with ``ones'' on the diagonal.
This is a $p$-Sylow subgroup, of order~$p^{\binom{k}{2}}$, and the
Solomon--Tits theorem also says that the restriction along $U_k(\bZ/p)
\subset GL_k(\bZ/p)$ of the Steinberg representation is the regular
integral representation.

\begin{addendum} \label{add:VGBequiv}
The homotopy cofiber sequence~\eqref{eq:deltaseq} is
$GL_k(\bZ/p)$-equivariant, and the equivalence in Theorem~\ref{thm:VGB}(a)
can be written $U_k(\bZ/p)$-equivariantly as
$$
V(G, B)^\wedge_p \simeq U_k(\bZ/p)_+ \wedge \Sigma^{-k} B^\wedge_p \,.
$$
\end{addendum}

\begin{proof}
As reviewed above, the Solomon--Tits equivalence can be written
$U_k(\bZ/p)$-equivariantly as $\Sigma^2 \cT_k \simeq U_k(\bZ/p)_+ \wedge
S^k$, and the rest of the analysis respects this action.
\end{proof}

As stated at the beginning of this section, we assume throughout the
remainder of the paper that Theorem~\ref{thm:segal-smash-pgroup} holds for
each proper subquotient of~$G$.  In particular, the inductive hypotheses
in Proposition~\ref{prop:SinftyrhoBGG} and Theorem~\ref{thm:VGB} are
satisfied when $G$ is a $p$-group.

\section{Towers of extended powers}
\label{sec:towers}

In the papers~\cite{Car84}*{\S III}, \cite{CMP87}*{\S 8},
the spectrum $F(S^{\infty\rho}, EG_+ \wedge X)^G$ is analyzed under the
hypothesis that the $G$-spectrum~$X$ is split.  This enables a translation
into non-equivariant terms, involving the $X$-homology of a tower
$$
\dots \longto BG^{-(m+1)\rho} \longto BG^{-m\rho} \longto \dots
\longto BG_+
$$
of Thom spectra.  The smash power $G$-spectra $X = B^{\wedge G}$
are not generally split.  (For example, with $B = H = H\bF_p$
and $G = C_p$ we have $\pi_0((B^{\wedge G})^G) \cong \bZ/p^2$ by a
variant of~\cite{HM97}*{Thm.~3.3}, and this group does not contain
$\pi_0(B^{\wedge G}) \cong \bF_p$ as a direct summand.)  We shall
therefore instead follow Steenrod~\cite{Ste62} and calculate with the
mod~$p$ cohomology of the tower
\begin{equation} \label{eq:extpowtower}
\dots \longto \Sigma^{2(m+1)} D_G(\Sigma^{-2(m+1)} B)
	\overset{\Delta}\longto \Sigma^{2m} D_G(\Sigma^{-2m} B) \longto \dots
	\longto D_G B
\end{equation}
of $G$-fold extended power spectra.  Recall that we write $W(G, B)$
for the homotopy limit of this tower.

Let $p$ be a prime, briefly write $H_*(-) = H_*(-; \bF_p)$ and $H^*(-)
= H^*(-; \bF_p)$, and let $L = e(\rho) \in H_{gp}^{2(|G|-1)}(G)$ be
the mod~$p$ Euler class of the $G$-representation~$\rho$.
If $G = (C_p)^k$ is elementary abelian and $p$ is odd, then its group
cohomology
$$
H_{gp}^*(G) = E(x_1, \dots, x_k) \otimes P(y_1, \dots, y_k)
$$
is a tensor product of exterior and polynomial algebras, with $|x_i|
= 1$, $|y_i| = 2$ and $\beta(x_i) = y_i$ for each $1 \le i \le k$.
If instead $p=2$ then
$$
H_{gp}^*(G) = P(x_1, \dots, x_k)
$$
with $|x_i| = 1$ and $\beta(x_i) = x_i^2$ for each~$1 \le i \le k$.
In either case
$$
L = \prod_{x\ne0} \beta(x) \,,
$$
where $x$ ranges over the $(p^k-1)$ nonzero elements in $H_{gp}^1(G) =
\bF_p\{x_1, \dots, x_k\}$.  Let~$A$ denote the mod~$p$ Steenrod algebra.

\begin{proposition} \label{prop:cohomWGB}
Let $G$ be a $p$-group and $B$ a bounded below spectrum with $H_*(B;
\bF_p)$ of finite type.

(a)
There is a natural $A$- and $H_{gp}^*(G)$-linear isomorphism
$$
H^*(D_G B)[L^{-1}] \cong \colim_m H^*(\Sigma^{2m} D_G(\Sigma^{-2m} B))
	=: H_c^*(W(G, B)) \,.
$$

(b)
There is an $H_{gp}^*(G)$-linear isomorphism
$$
H^*(D_G B)[L^{-1}]
	\cong H_{gp}^*(G)[L^{-1}] \{ b^{\otimes G} \mid b \in \cB \} \,,
$$
where $\cB$ is a homogeneous basis for $H^*(B)$ and $b^{\otimes G}$
denotes the tensor product of one copy of $b$ for each element $g \in G$.

(c)
If $G$ is not elementary abelian, then $H^*(D_G B)[L^{-1}] = 0$.

(d)
If $G = (C_p)^k$ then the isomorphisms in~(a) and~(b) are
$GL_k(\bZ/p)$-linear.
\end{proposition}

\begin{proof}
(a)
For each $m\ge0$ there is a Thom isomorphism $H^*(\Sigma^{2m}
D_G(\Sigma^{-2m} B)) \cong H^{* + 2m(|G|-1)}(D_B G)$ under which the
homomorphism induced by $\Delta$ is given by multiplication by~$L$,
cf.~\cite{BMMS86}*{Lem.~II.5.6}.  Hence the continuous cohomology
$H_c^*(W(G, B))$ is isomorphic to the colimit of the sequence
$$
\dots \longfrom H^{* + 2(m+1)(|G|-1)}(D_G B)
	\overset{L}\longfrom H^{* + 2m(|G|-1)}(D_G B)
	\longfrom \dots \longfrom H^*(D_G B) \,,
$$
i.e., the localization of $H^*(D_G B)$ away from the Euler class~$L$.

(b)
Following Steenrod~\cite{Ste62}*{\S VIII.3}, we have a natural isomorphism
$$
H^*(D_G B) \cong H_{gp}^*(G; H^*(B)^{\otimes G}) \,.
$$
A basis for $H^*(B)^{\otimes G}$ is given by the tensor products $b'
= \otimes_{g \in G} b_g$, where each $b_g \in \cB$ lies in the chosen
basis, and the action by $G$ permutes these generators (up to signs).
Hence $H_{gp}^*(G; H^*(B)^{\otimes G})$ splits as a direct sum of summands
$$
H_{gp}^*(G; \bF_p[G/K]\{b'\}) \cong H_{gp}^*(K; \bF_p\{b'\}) \,,
$$
where $K$ is the stabilizer of~$b'$.  If $K$ is a proper subgroup
of~$G$, then $L$ restricts trivially to $H_{gp}^*(K)$, and the summand
$H_{gp}^*(G; \bF_p[G/K]\{b'\})$ is annihilated by localization away
from~$L$.  Only the summands with $b' = b^{\otimes G}$ survive, each of
which contributes $H_{gp}^*(G)[L^{-1}]\{b'\}$ to $H^*(D_G B)[L^{-1}]$.

(c)
If $G$ is not elementary abelian then $L \in H_{gp}^*(G)$ is nilpotent
by the Quillen--Venkov theorem~\cite{QV72}, \cite{Car84}*{Lem.~III.1},
hence $H_{gp}^*(G)[L^{-1}] = 0$.

(d)
The action of each element in $GL_k(\bZ/p)$ permutes the elements
in $G = (C_p)^k$, hence also permutes the tensor factors in
$b^{\otimes G}$, all of which are equal.
\end{proof}

To pass from continuous cohomology to homotopy groups, we make
use of an inverse limit of Adams spectral sequences associated
to the tower~\eqref{eq:extpowtower}, as in~\cite{Car84}*{\S III},
~\cite{CMP87}*{\S7} and~\cite{LNR12}*{\S2}.

\begin{proposition} \label{prop:limitAdamsspseq}
Let $G$ be a $p$-group and $B$ a bounded below spectrum with $H_*(B;
\bF_p)$ of finite type.  There is a natural, strongly convergent,
inverse limit Adams spectral sequence
$$
E_2^{s,t} = \Ext_A^{s,t}(H^*(D_G B)[L^{-1}], \bF_p)
	\Longrightarrow \pi_{t-s} W(G, B)^\wedge_p \,.
$$
\end{proposition}

\begin{proof}
Each spectrum $\Sigma^{2m} D_G(\Sigma^{-2m} B)$ is bounded below with
mod~$p$ homology of finite type.  Its mod~$p$ Adams spectral
sequence
$$
E_2^{*,*}(m) = \Ext_A^{*,*}(H^*(\Sigma^{2m} D_G(\Sigma^{-2m} B)), \bF_p)
        \Longrightarrow \pi_* \Sigma^{2m} D_G(\Sigma^{-2m} B)^\wedge_p
$$
is therefore strongly convergent, with $E_2^{*,*}(m)$ finite in each
bidegree.  By~\cite{CMP87}*{Prop.~7.1}, in the slightly generalized form
from~\cite{LNR12}*{Prop.~2.2}, the algebraic limit groups $E_r^{*,*} =
\lim_m E_r^{*,*}(m)$ (and the induced differentials $d_r$) also form a
spectral sequence, with $E_2$-term
\begin{align*}
E_2^{*,*} &= \lim_m \Ext_A^{*,*}(H^*(\Sigma^{2m} D_G(\Sigma^{-2m} B)), \bF_p) \\
&\cong \Ext_A^{*,*}(\colim_m H^*(\Sigma^{2m} D_G(\Sigma^{-2m} B)), \bF_p) \\
&= \Ext_A^{*,*}(H_c^*(W(G, B)), \bF_p)
	\cong \Ext_A^{*,*}(H^*(D_G B)[L^{-1}], \bF_p)\,.
\end{align*}
Moreover, this spectral sequence converges strongly to
$$
\pi_* \holim_m \Sigma^{2m} D_G(\Sigma^{-2m} B)^\wedge_p
	= \pi_* W(G, B)^\wedge_p \,,
$$
as asserted.
\end{proof}

\begin{proposition} \label{prop:WGBnotelemab}
Let $G$ be a $p$-group that is not elementary abelian, and let $B$ be a
bounded below spectrum with $H_*(B; \bF_p)$ of finite type.
Then $W(G, B)^\wedge_p \simeq *$.
\end{proposition}

\begin{proof}
By Proposition~\ref{prop:cohomWGB}(c) we have $H^*(D_G B)[L^{-1}] = 0$,
so $E_2^{*,*} = 0$ in the inverse limit spectral sequence of
Proposition~\ref{prop:limitAdamsspseq}, which by strong convergence
implies $\pi_* W(G, B)^\wedge_p = 0$.
\end{proof}

We can now collect some of the threads, as in the proof
of~\cite{Car84}*{Thm.~C}.

\begin{proof}[Proof of Theorem~\ref{thm:segal-smash-pgroup} for $G$
not elementary abelian]
Let $G$ be a $p$-group that is not elementary abelian, and suppose that
Theorem~\ref{thm:segal-smash-pgroup} holds for each proper subquotient
of~$G$.  Let $B$ be a bounded below flat orthogonal spectrum with
$H_*(B; \bF_p)$ of finite type.  By Theorem~\ref{thm:VGB}(b), $V(G,
B)^\wedge_p \simeq *$.  By Proposition~\ref{prop:WGBnotelemab},
$W(G, B)^\wedge_p \simeq *$.  Hence $F(S^{\infty\rho}, B^{\wedge
G})^G$ becomes trivial after $p$-completion, by the homotopy cofiber
sequence in Definition~\ref{def:VWdelta}.  Therefore $\gamma \: (B^{\wedge G})^G
\to (B^{\wedge G})^{hG}$ becomes an equivalence after $p$-completion,
by Proposition~\ref{prop:SinftyrhoBGG}.
\end{proof}

In the elementary abelian case, with $G \cong (C_p)^k$, we need better
control of the connecting map~$\delta$.  Suppose that $B$ is bounded
below with $H_*(B; \bF_p)$ of finite type.  Then $V(G, B)^\wedge_p$ and
$\Sigma^{2m} D_G(\Sigma^{-2m} B)$ are also bounded below with mod~$p$
homology of finite type, in view of our inductive hypothesis on~$G$
and Theorem~\ref{thm:VGB}(a).  For each~$m\ge0$ the composite map
$$
V(G, B) \overset{\delta}\longto W(G, B)
	\longto \Sigma^{2m} D_G(\Sigma^{-2m} B)
$$
induces a morphism
$$
\Ext_A^{*,*}(H^*(V(G, B)), \bF_p)
	\longto \Ext_A^{*,*}(H^*(\Sigma^{2m} D_G(\Sigma^{-2m} B)), \bF_p)
	= E_2^{*,*}(m)
$$
of strongly convergent Adams spectral sequences.  Passing to the limit
over~$m$, these define a natural morphism of spectral sequences
$$
E_2(\delta) \: \Ext_A^{*,*}(H^*(V(G, B)), \bF_p)
	\longto \Ext_A^{*,*}(H_c^*(W(G, B)), \bF_p) = E_2^{*,*}
$$
converging to $\pi_*(\delta)^\wedge_p \: \pi_* V(G, B)^\wedge_p
\to W(G, B)^\wedge_p$.  By construction, $E_2(\delta) = f_B^*$ is induced
by the homomorphism $f_B$ specified in the following definition.

\begin{definition} \label{def:fB}
Let $f_B = \delta^* \kappa$ be the natural $A$- and $\Aut(G)$-linear
homomorphism defined by the composition
$$
f_B \: H_c^*(W(G, B))
	\overset{\kappa}\longto H^*(W(G, B))
	\overset{\delta^*}\longto H^*(V(G, B)) \,,
$$
where
$$
\kappa \: H_c^*(W(G, B))
	= \colim_m H^*(\Sigma^{2m} D_G(\Sigma^{-2m} B))
	\longto H^*(W(G, B))
$$
is the canonical map from the continuous to the ordinary mod~$p$
cohomology associated to the tower~\eqref{eq:extpowtower}.
\end{definition}

\section{Comparison of $\Tor^A$-equivalences}

We now assume that $G = (C_p)^k$ with $k\ge1$, so that $\Aut(G) =
GL_k(\bZ/p)$, and that Theorem~\ref{thm:segal-smash-pgroup} holds for
each proper subquotient of~$G$.  We will show that
$$
\delta^\wedge_p \: V(G, B)^\wedge_p \longto W(G, B)^\wedge_p
$$
is an equivalence for suitable~$B$ by using the classical Segal conjecture
to show that the $A$- and $GL_k(\bZ/p)$-linear homomorphism
$$
f_B \: H^*(D_G B)[L^{-1}] \cong H_c^*(W(G, B))
	\longto H^*(V(G, B))
	\cong \bigoplus^{p^{\binom{k}{2}}} \Sigma^{-k} H^*(B) \,,
$$
cf.~Definition~\ref{def:fB}, Proposition~\ref{prop:cohomWGB}(a,d) and
Theorem~\ref{thm:VGB}(a), is a $\Tor^A$-equivalence in the key cases $B =
S$ and $B = H = H\bF_p$.  More precisely, as a $U_k(\bZ/p)$-module the
target can be rewritten as
$$
H^*(V(G, B)) \cong \Sigma^{-k} H^*(B)[U_k(\bZ/p)] \,,
$$
cf.~Addendum~\ref{add:VGBequiv}.  The proof will be an application of
the following comparison theorem of Priddy--Wilkerson.
(As an aside, we recall that for $p$-groups~$U$ an $\bF_p[U]$-module
is projective if and only if it is free.)

\begin{theorem}[\cite{PW85}*{Thm.~III(i)}]
\label{thm:PWcomparison}
Let $A$ be the mod~$p$ Steenrod algebra, let $U$ be a $p$-group, and
let $A[U] = A \otimes \bF_p[U]$ denote the group algebra.  Let $f \:
M \to N$ be a surjective $A[U]$-module homomorphism, where $M$ and $N$
are projective as $\bF_p[U]$-modules.  If
$$
f^U_* \: \Tor^A_{*,*}(\bF_p, M^U) \longto \Tor^A_{*,*}(\bF_p, N^U)
$$
is an isomorphism then
$$
f_* \: \Tor^A_{*,*}(\bF_p, M) \longto \Tor^A_{*,*}(\bF_p, N)
$$
is an isomorphism, too.
\end{theorem}

Recall from~\cite{AGM85}*{Prop.~1.2} that the conclusion about~$f_*$,
i.e., that it is a $\Tor^A$-equivalence, also implies that
\begin{equation} \label{eq:Extiso}
f^* \: \Ext_A^{*,*}(N, Q) \longto \Ext_A^{*,*}(M, Q)
\end{equation}
is an isomorphism for each $A$-module~$Q$ that is bounded below and
of finite type.  This will be applied with $Q = \bF_p$ to show that a
morphism of Adams spectral sequences is an isomorphism.

For brevity we hereafter set
\begin{equation}
U := U_k(\bZ/p) \subset GL_k(\bZ/p) \,.
\end{equation}

\begin{remark}
Our approach to specifying $f_B$ differs from that of~\cite{PW85}*{(1.7)},
where $f$ for $B = S$ is instead defined via the evidently surjective
projection
$$
f \: H_{gp}^*(G)[L^{-1}] \cong H_c^*(W(G, S))
	\longto \bF_p \otimes_A H_{gp}^*(G)[L^{-1}]
$$
onto the $A$-module coinvariants, followed by an (a posteriori)
identification of the target with $\Sigma^{-k} \St_k \otimes \bF_p \cong
H^*(V(G, S))$.  This will not work for many other spectra~$B$, including
$B = H$, since $A$ generally acts non-trivially on $H^*(V(G, B))$.
\end{remark}

To verify that our homomorphism $f_B$ is surjective for $B = H$,
we now rely on the classical Segal conjecture in the case $B = S$,
including the delicate comparison in~\cite{CMP87}*{\S\S 5--6} of $S_G$
with $F(EG_+, H)$, representing stable equivariant cohomotopy and mod~$p$
Borel cohomology, respectively.

\begin{proposition} \label{prop:fSsurj}
$f_S \: H_{gp}^*(G)[L^{-1}] \to \Sigma^{-k} \bF_p[U]$ is surjective.
\end{proposition}

\begin{proof}
The edge homomorphism of the 
inverse limit Adams spectral sequence
$$
E_2^{*,*} = \Ext_A^{*,*}(H_c^*(W(G, S)), \bF_p)
	\Longrightarrow \pi_* W(G, S)^\wedge_p
$$
from Proposition~\ref{prop:limitAdamsspseq}
factors as
$$
\pi_* W(G, S)^\wedge_p
	\overset{h}\longto \Hom_A(H^*(W(G, S)), \bF_p)
	\overset{\kappa^*}\longto
	\Hom_A(H_c^*(W(G, S)), \bF_p) \,,
$$
where $h$ is induced by the Hurewicz homomorphism.
By Adams--Gunawardena--Miller~\cite{AGM85}*{Thm.~1.1(a,b)},
there is a $\Tor^A$-equivalence
$$
H_c^*(W(G, S)) \cong H_{gp}^*(G)[L^{-1}]
	\longto \bigoplus^{p^{\binom{k}{2}}} \Sigma^{-k} \bF_p \,.
$$
Hence the inverse limit Adams spectral sequence is isomorphic to
the direct sum of $p^{\binom{k}{2}}$ copies of the Adams spectral
sequence for~$S^{-k}$, and this implies that $W(G, S)^\wedge_p \simeq
\bigvee^{p^{\binom{k}{2}}} \Sigma^{-k} S^\wedge_p$.  In particular,
the homomorphisms $\bar h$ and $\kappa^* \bar h$ in
$$
\pi_{-k} W(G, S)/p
	\overset{\bar h}\longto \Hom_A^{-k}(H^*(W(G, S)), \bF_p)
	\overset{\kappa^*}\longto \Hom_A^{-k}(H_c^*(W(G, S)), \bF_p)
$$
are isomorphisms, hence so is~$\kappa^*$.  It follows that $\kappa$
is surjective.  Finally, $\delta^\wedge_p$ is an equivalence for $B =
S$, by the classical Segal conjecture, so $\delta^* \: H^*(W(G, S))
\to H^*(V(G, S))$ is an isomorphism.  Hence $f_S = \delta^* \kappa$
is surjective.
\end{proof}

\begin{proposition} \label{prop:fHsurj}
$f_H \: H^*(D_G H)[L^{-1}] \to \Sigma^{-k} A[U]$ is surjective.
\end{proposition}

\begin{proof}
By naturality of $f_B$ with respect to the mod~$p$
Hurewicz map $h \: S \to H$, the $A$-module diagram
$$
\xymatrix{
H^*(D_G H)[L^{-1}] \cong H_c^*(W(G, H)) \ar[r]^-{f_H} \ar[d]_-{h^*}
& H^*(V(G, H)) \cong \Sigma^{-k} A[U] \ar[d]^-{h^*} \\
H_{gp}^*(G)[L^{-1}] \cong H_c^*(W(G, S)) \ar[r]^-{f_S}
& H^*(V(G, S)) \cong \Sigma^{-k} \bF_p[U]
}
$$
commutes, where $A = H^*(H)$ and~$\bF_p = H^*(S)$.  The left
hand homomorphism~$h^*$ and $f_S$ are both surjective, by
Propositions~\ref{prop:cohomWGB}(b) and~\ref{prop:fSsurj}.  Hence the
image of $f_H$ contains all of the $A$-module generators of $\Sigma^{-k}
A[U]$, in degree~$-k$, which by $A$-linearity implies that $f_H$
is surjective.
\end{proof}

The projectivity hypothesis in Theorem~\ref{thm:PWcomparison} follows
easily from the special case considered by Priddy--Wilkerson.

\begin{proposition} \label{prop:projective}
Let $B$ be bounded below with $H_*(B; \bF_p)$ of finite type.  Then
$H^*(D_G B)[L^{-1}]$ and $\Sigma^{-k} H^*(B)[U]$ are both projective
as $\bF_p[U]$-modules.
\end{proposition}

\begin{proof}
Suppose $p$ is odd.  By~\cite{PW85}*{Prop.~2.4} the homomorphism
$$
P(y_1, \dots, y_k)[L^{-1}]^{GL_k(\bZ/p)} \longto P(y_1, \dots, y_k)[L^{-1}]
$$
is a $GL_k(\bZ/p)$-Galois extension of commutative rings, in the sense
of~\cite{CHR65}*{Thm.~1.3, Def.~1.4}.  It follows from~\cite{CHR65}*{Thm.~2.2,
Thm.~4.2(a)} that
$$
P(y_1, \dots, y_k)[L^{-1}]^U \longto P(y_1, \dots, y_k)[L^{-1}]
$$
is a $U$-Galois extension, so that $P(y_1, \dots, y_k)[L^{-1}]$ is a
projective $\bF_p[U]$-module.  Hence~\cite{PW85}*{Prop.~2.5} implies that
$$
H_{gp}^*(G)[L^{-1}] = E(x_1, \dots, x_k) \otimes P(y_1, \dots, y_k)[L^{-1}]
$$
and
$$
H^*(D_G B)[L^{-1}] \cong H_{gp}^*(G)[L^{-1}] \{b^{\otimes G} \mid b \in \cB\}
$$
are also projective as $\bF_p[U]$-modules.  The case $p=2$ is a little
simpler, replacing $P(y_1, \dots, y_k)$ by $P(x_1, \dots, x_k)$, omitting
the factor $E(x_1, \dots, x_k)$, and noting that our Euler class~$L$ is
the square of the class considered by Priddy--Wilkerson.  The claim for
$\Sigma^{-k} H^*(B)[U] \cong \bF_p[U] \{\Sigma^{-k} b \mid b \in \cB\}$
is immediate.
\end{proof}

Our next aim is to generalize results of Li--Singer and
Adams--Gunawardena--Miller to identify $H^*(D_G B)[L^{-1}]^U$ as
an $A$-module with the $k$-fold iterated desuspended $C_p$-Singer
construction $T^k(H^*(B))$, which is $\Tor^A$-equivalent to $\Sigma^{-k}
H^*(B)$.  The notation~$T(M)$, for any $A$-module~$M$, is that
of~\cite{AGM85}*{\S2}, and is equal to the $A$-module denoted $\Sigma^{-1}
R_+(M)$ in~\cite{LNR12}*{Def.~3.1}.  We shall use the expressions
\begin{align*}
T(H^*(B)) &\cong H_c^*(W(C_p, B))
	= \colim_m H^*(\Sigma^{2m} D_{C_p}(\Sigma^{-2m} B)) \\
	&\cong H^*(D_{C_p} B)[L_1^{-1}]
	\cong H_{gp}^*(C_p; H^*(B)^{\otimes p})[L_1^{-1}] \\
	&\cong H_{gp}^*(C_p)[L_1^{-1}] \{b^{\otimes p} \mid b \in \cB\}
\end{align*}
from~\cite{LNR12}*{Thm.~5.9}, extending~\cite{BMMS86}*{Thm.~II.5.1}, as
presentations of this version of the Singer construction.  Here $L_1 =
-\beta(x_1)^{p-1} \in H_{gp}^{2(p-1)}(C_p)$ is the case $k=1$ of the
Euler class~$L$.

\begin{definition}
For any group~$H$, let $C_p \wr H = C_p \ltimes H^p$ denote
the wreath product, i.e., the semidirect product where $C_p$
acts on the $p$-th power $H^p$ by cyclically permuting the factors.
Let $C_p \times H \to C_p \wr H$ be the diagonal
inclusion mapping $(g, h)$ to $(g; h, \dots, h)$, and
let
$$
d \: (C_p)^k = C_p \times \dots \times C_p
	\longto C_p \wr \dots \wr C_p = \wr^k C_p
$$
denote its $(k-1)$-fold iterate, with $H = \wr^i C_p$ at the $i$-th
instance.
\end{definition}

\begin{lemma}
View $d$ as an inclusion of subgroups of $\Sigma_{p^k}$.  The normalizer
of $G = (C_p)^k$ in the $p$-Sylow subgroup $\wr^k C_p$ of $\Sigma_{p^k}$
is $U_k(\bZ/p) \ltimes (C_p)^k$, with Weyl group the $p$-Sylow subgroup
$U = U_k(\bZ/p)$ of $GL_k(\bZ/p)$.
\end{lemma}

\begin{proof}
The normalizer of $(C_p)^k$ in $\Sigma_{p^k}$ is $GL_k(\bZ/p) \ltimes
(C_p)^k$, and the lemma follows by restricting to elements in the
$p$-Sylow subgroup $\wr^k C_p$ of $\Sigma_{p^k}$.
\end{proof}

The diagonal inclusion~$d$ induces a natural map of extended powers
$$
D_{(C_p)^k} B \simeq E\Sigma_{p^k+} \wedge_{(C_p)^k} B^{\wedge p^k}
	\longto
E\Sigma_{p^k+} \wedge_{\wr^k C_p} B^{\wedge p^k}
	\simeq D_{C_p}(\cdots D_{C_p}(B) \cdots) \,,
$$
and a morphism of collapsing~\cite{May70}*{Lem.~1.1(iii)} homotopy orbit
spectral sequences, from
$$
{}' E_2^{*,*} = H_{gp}^*(\wr^k C_p; H^*(B)^{\otimes p^k})
	\Longrightarrow H^*(E\Sigma_{p^k+} \wedge_{\wr^k C_p} B^{\wedge p^k})
$$
to
$$
{}'' E_2^{*,*} = H_{gp}^*((C_p)^k; H^*(B)^{\otimes p^k})
	\Longrightarrow H^*(E\Sigma_{p^k+} \wedge_{(C_p)^k} B^{\wedge p^k})
\,,
$$
given at the $E_2$-terms by the homomorphism
\begin{equation} \label{eq:dstar}
d^* \: H_{gp}^*(\wr^k C_p; H^*(B)^{\otimes p^k})
	\longto H_{gp}^*((C_p)^k; H^*(B)^{\otimes p^k}) \,.
\end{equation}
More generally, for each $m\ge0$ the diagonal inclusion induces a map
\begin{multline*}
d_{B,m} \: \Sigma^{2m} D_{(C_p)^k} \Sigma^{-2m} B
	\longto
	\Sigma^{2m} D_{C_p}( \cdots D_{C_p}(\Sigma^{-2m} B) \cdots) \\
	\simeq
	\Sigma^{2m} D_{C_p} \Sigma^{-2m}( \cdots \Sigma^{2m} D_{C_p} \Sigma^{-2m}(B) \cdots)
\end{multline*}
to the $k$-fold iterate of $\Sigma^{2m} D_{C_p} \Sigma^{-2m}(-)$ applied
to~$B$, and these are compatible under the twisted diagonal maps~$\Delta$.
Passing to cohomology we obtain homomorphisms
$$
d_{B,m}^* \: H^*(\Sigma^{2m} D_{C_p} \Sigma^{-2m}
	( \cdots \Sigma^{2m} D_{C_p}\Sigma^{-2m}(B) \cdots))
\longto H^*(\Sigma^{2m} D_{(C_p)^k} \Sigma^{-2m} B)
$$
factoring through the $U$-invariants of the target, and passing to
colimits over~$m$ we obtain an $A$-module homomorphism
\begin{align*}
d_B^* \: T^k(H^*(B))
&\cong \colim_{m_1, \dots, m_k} H^*(\Sigma^{2m_1} D_{C_p} \Sigma^{-2m_1}(
	\cdots \Sigma^{2m_k} D_{C_p} \Sigma^{-2m_k}(B) \cdots)) \\
&\cong \colim_m H^*(\Sigma^{2m} D_{C_p} ( \cdots D_{C_p}(\Sigma^{-2m}
	B) \cdots)) \\
&\longto \colim_m H^*(\Sigma^{2m} D_{(C_p)^k} \Sigma^{-2m} B)
	= H^*(D_G B)[L^{-1}]
\end{align*}
from the $k$-fold iterate of $T$ applied to $H^*(B)$, with image
contained in the $U$-invariants of the target.  Here we use that the
Singer construction $T$ commutes with sequential colimits, and that the
$k$-tuples $(m, \dots, m)$ are cofinal among the $(m_1, \dots, m_k)$.

\begin{proposition} \label{prop:TkHBisoUinv}
Let $B$ be bounded below with $H_*(B; \bF_p)$ of finite type.  The
homomorphism $d_B^*$ factors through a natural isomorphism of $A$-modules
$$
T^k(H^*(B)) \overset{\cong}\longto
	H^*(D_G B)[L^{-1}]^U \,.
$$
\end{proposition}

\begin{proof}
In the case $B = S$, William Singer's result \cite{Sin81}*{Prop.~9.1} (for
$p=2$) and the extension \cite{AGM85}*{Thm.~1.4} of~\cite{LS82}*{(1.6)}
(for $p$ odd) show that
$$
d_S^* \: T^k(\bF_p) \overset{\cong}\longto H_{gp}^*(G)[L^{-1}]^U
	\subset H_{gp}^*(G)[L^{-1}]
$$
maps $T^k(\bF_p)$ isomorphically
to the $U$-invariants of $H_{gp}^*(G)[L^{-1}]$.
In general, we can rewrite the lift of $d_B^*$ as
\begin{multline*}
H_{gp}^*(C_p; (\cdots H_{gp}^*(C_p; H^*(B)^{\otimes p})[L_1^{-1}]
	\cdots)^{\otimes p})[L_1^{-1}] \\
\longto H_{gp}^*((C_p)^k; H^*(B)^{\otimes p^k})[L^{-1}]^U \,,
\end{multline*}
or as
$$
T^k(\bF_p)
\{b^{\otimes p^k} \mid b \in \cB\}
\longto
H_{gp}^*(G)[L^{-1}]^U
\{b^{\otimes p^k} \mid b \in \cB\} \,.
$$
Here both sides are filtered by the cohomological degree of $b^{\otimes
p^k}$, and the homomorphism respects these filtrations.  Then,
just as in~\eqref{eq:dstar},
$$
d_B^*(x \cdot b^{\otimes p^k}) \equiv d_S^*(x) \cdot b^{\otimes p^k}
\quad\text{modulo lower filtrations}
$$
for $x \in T^k(\bF_p)$ and $b \in H^q(B)$, and it follows by induction
on the degree~$q$ that the lift of~$d_B^*$ is an isomorphism.
\end{proof}

\begin{proposition} \label{prop:fHUToreq}
$$
f_H^U \: H^*(D_G H)[L^{-1}]^U
	\longto \Sigma^{-k} A[U]^U \cong \Sigma^{-k} A
$$
is a $\Tor^A$-equivalence.
\end{proposition}

\begin{proof}
For any $A$-module~$M$, the $\Tor^A$-equivalence $\epsilon \: T(M)
\to \Sigma^{-1} M$ of~\cite{AGM85}*{Thm.~1.3} can be iterated $k$-fold
to give a $\Tor^A$-equivalence $\epsilon^k \: T^k(M) \to \Sigma^{-k} M$.
When combined with Proposition~\ref{prop:TkHBisoUinv}, this gives a
$\Tor^A$-equivalence
$$
H^*(D_G B)[L^{-1}]^U \cong
	T^k(H^*(B))
	\overset{\epsilon^k}\longto \Sigma^{-k} H^*(B) \,.
$$
It follows that the source and target of $f_B^U$ are abstractly
$\Tor^A$-equivalent, but it remains to verify, in the special case $B
= H$, that $f_H^U$ induces this equivalence.  Using~\eqref{eq:Extiso}
with $Q = \Sigma^{-k} A$ we see that
$$
\Hom_A(H^*(D_G H)[L^{-1}]^U, \Sigma^{-k} A)
\cong
\Hom_A(\Sigma^{-k} A, \Sigma^{-k} A) \cong \bF_p \,,
$$
so that $f_H^U$ is a multiple in~$\bF_p$ times a $\Tor^A$-equivalence.
The conclusion now follows, since Propositions~\ref{prop:fHsurj}
and~\ref{prop:projective} imply that $f_H^U$ is nonzero.
\end{proof}

\section{The elementary abelian case}
\label{sec:elemabel}

We continue to assume that $G = (C_p)^k$ with $k\ge1$, and that
Theorem~\ref{thm:segal-smash-pgroup} holds for each proper subquotient
of~$G$.

\begin{proposition} \label{prop:fHToreq}
$f_H \: H^*(D_G H)[L^{-1}] \to \Sigma^{-k} A[U]$ is a
$\Tor^A$-equivalence.
\end{proposition}

\begin{proof}
This follows from the Priddy--Wilkerson comparison
theorem, i.e., Theorem~\ref{thm:PWcomparison},
for the $A[U]$-module homomorphism $f = f_H$, since $f_H$ is
surjective by Proposition~\ref{prop:fHsurj}, its source and target
are $\bF_p[U]$-projective by Proposition~\ref{prop:projective},
and its $U$-invariant part $f_H^U$ is a $\Tor^A$-equivalence by
Proposition~\ref{prop:fHUToreq}.
\end{proof}

\begin{theorem} \label{thm:deltaforHeqce}
$\delta^\wedge_p \: V(G, H)^\wedge_p \to W(G, H)^\wedge_p$
is an equivalence.
\end{theorem}

\begin{proof}
By~Proposition~\ref{prop:limitAdamsspseq} and Definition~\ref{def:fB}
we have a morphism
$$
E_2(\delta) = f_H^* \: \Ext_A^{*,*}(\Sigma^{-k} A[U], \bF_p)
	\longto \Ext_A^{*,*}(H^*(D_G H)[L^{-1}], \bF_p)
$$
of Adams and inverse limit Adams spectral sequences, converging to
the homomorphism
$$
\pi_*(\delta^\wedge_p) \: \Sigma^{-k} \bF_p[U]
	\cong \pi_* V(G, H)^\wedge_p \longto \pi_* W(G, H)^\wedge_p \,.
$$
Here $f_H^*$ is an isomorphism by Proposition~\ref{prop:fHToreq}
and~\eqref{eq:Extiso}, which implies that $\pi_*(\delta^\wedge_p)$
is an isomorphism, as claimed.
\end{proof}

This proves Theorem~\ref{thm:segal-smash-pgroup} for $G = (C_p)^k$
and $B = H$.  Given this toehold result, we can deduce the theorem
for bounded below $B$ with $H_*(B; \bF_p)$ of finite type by the
inductive strategy of Nikolaus--Scholze~\cite{NS18}*{\S III.1}.
Recall that $\cT_k$ denotes the rank~$k$ Tits building, which is
a finite complex.

\begin{proposition} \label{prop:htpycofibseq}
Let $B' \to B \to B''$ be a homotopy cofiber sequence of bounded below
spectra with mod~$p$ homology of finite type.  Then
$$
V(G, B')^\wedge_p
	\longto V(G, B)^\wedge_p
	\longto V(G, B'')^\wedge_p
$$
and
$$
W(G, B')
	\longto W(G, B)
	\longto W(G, B'')
$$
are homotopy cofiber sequences.
\end{proposition}

\begin{proof}
It is clear that
$$
F(\Sigma^2 \cT_k, B')
	\longto F(\Sigma^2 \cT_k, B)
	\longto F(\Sigma^2 \cT_k, B'')
$$
is a homotopy cofiber sequence.  This implies the corresponding
result for $V(G, -)^\wedge_p$ by Theorem~\ref{thm:VGB}.

For the second claim we follow the proof of~\cite{BMMS86}*{Prop.~II.3.11}.
We may assume $B' \to B$ is a cofibration, with $B/B' = B''$.
There is a filtration
$$
D_G(B') = \Gamma^{p^k}(B) \to \dots \to \Gamma^{i+1}(B) \to \Gamma^i(B)
	\to \dots \to \Gamma^0(B) = D_G(B)
$$
with quotients
$$
\Gamma^i(B)/\Gamma^{i+1}(B) \simeq EG_+ \wedge_G \bigvee^{\binom{p^k}{i}}
	(B')^{\wedge i} \wedge (B'')^{\wedge p^k-i} \,.
$$
For $i=0$, this is $D_G(B'')$.  For $0<i<p^k$ it is a finite wedge sum
of spectra, each of the form
$$
EG_+ \wedge_K (B')^{\wedge i} \wedge (B'')^{\wedge p^k-i}
$$
with $K \subset G$ a proper subgroup.  These filtrations and splittings
are compatible with the twisted diagonal maps~$\Delta$.  Since the
inclusion $e \: S^0 \to S^\rho$ is $K$-equivariantly null-homotopic,
it follows that
$$
\holim_m \frac{\Sigma^{2m} \Gamma^i(\Sigma^{-2m} B)}
	{\Sigma^{2m} \Gamma^{i+1}(\Sigma^{-2m} B)} \simeq {*}
$$
for each $0<i<p^k$.  Hence $W(G, B)/W(G, B') \to W(G, B'')$ is an
equivalence.
\end{proof}

\begin{proposition} \label{prop:Postnikov}
Let $B$ be a bounded below spectrum, with Postnikov tower
$$
B \to \cdots \to \tau_{\le {n+1}} B \to \tau_{\le n} B \to \cdots \,.
$$
The natural maps
$$
V(G, B)^\wedge_p \overset{\simeq}\longto
	\holim_n V(G, \tau_{\le n} B)^\wedge_p
$$
and
$$
W(G, B) \overset{\simeq}\longto
	\holim_n W(G, \tau_{\le n} B)
$$
are equivalences.
\end{proposition}

\begin{proof}
It is clear that
$$
F(\Sigma^2 \cT_k, B) \overset{\simeq}\longto
        \holim_n F(\Sigma^2 \cT_k, \tau_{\le n} B)
$$
is an equivalence, and, by Theorem~\ref{thm:VGB}, this implies the
corresponding result for $V(G, -)^\wedge_p$.

For the second claim, note that since~$B$ is bounded below, the
connectivity of
$$
\Sigma^{2m} D_G(\Sigma^{-2m} B)
	\longto \Sigma^{2m} D_G(\Sigma^{-2m} \tau_{\le n} B)
$$
increases to infinity with~$n$, for each fixed $m\ge0$.
Hence
$$
\Sigma^{2m} D_G(\Sigma^{-2m} B) \overset{\simeq}\longto
	\holim_n \Sigma^{2m} D_G(\Sigma^{-2m} \tau_{\le n} B)
$$
is an equivalence.  The result follows by passing to the homotopy
limit over~$m$ and interchanging the order of the two homotopy limits.
\end{proof}

\begin{theorem} \label{thm:deltaforBeqce}
Let $G = (C_p)^k$ with $k\ge1$, suppose that
Theorem~\ref{thm:segal-smash-pgroup} holds for each proper subquotient
of~$G$, and let $B$ be a flat orthogonal spectrum with $\pi_*(B)$
bounded below and $H_*(B; \bF_p)$ of finite type.  Then
$$
\delta^\wedge_p \: V(G, B)^\wedge_p \to W(G, B)^\wedge_p
$$
is an equivalence.
\end{theorem}

\begin{proof}
It suffices to prove that $\delta/p \: V(G, B)/p \to W(G, B)/p$ is
an equivalence.  In view of Proposition~\ref{prop:htpycofibseq}
and the homotopy cofiber sequence
$$
B \overset{p}\longto B \longto B/p \,,
$$
this is equivalent to checking that $\delta^\wedge_p$ for $B/p$ is an
equivalence.  Each Postnikov section $\tau_{\le n}(B/p)$ has only finitely
many nonzero homotopy groups, each of order a finite power of~$p$.
Hence the result for $\tau_{\le n}(B/p)$ follows by induction from
Theorem~\ref{thm:deltaforHeqce} and Proposition~\ref{prop:htpycofibseq}.
The result for $B/p$ then follows from Proposition~\ref{prop:Postnikov}.
\end{proof}

\begin{proof}[Proof of Theorem~\ref{thm:segal-smash-pgroup} for $G$
elementary abelian]
Let $G = (C_p)^k$ with $k\ge1$, suppose that
Theorem~\ref{thm:segal-smash-pgroup} holds for each proper subquotient
of~$G$, and let $B$ be a bounded below flat orthogonal spectrum with
$H_*(B; \bF_p)$ of finite type.  By Theorem~\ref{thm:deltaforBeqce},
the map $\delta^\wedge_p \: V(G, B)^\wedge_p \to W(G, B)^\wedge_p$
is an equivalence.  Hence $F(S^{\infty\rho}, B^{\wedge G})^G$ becomes
trivial after $p$-completion, by the homotopy cofiber sequence in
Definition~\ref{def:VWdelta}.  Therefore $\gamma \: (B^{\wedge G})^G
\to (B^{\wedge G})^{hG}$ becomes an equivalence after $p$-completion,
by Proposition~\ref{prop:SinftyrhoBGG}.
\end{proof}

\section{The finite group case}
\label{sec:finite}

We now assume that $G$ is any finite group and that $B$ is a flat
orthogonal spectrum that is bounded below and of finite type.
We aim to prove Theorem~\ref{thm:segal-smash-finite} concerning
the $G$-spectrum $X = B^{\wedge G}$.

Let $p$ be a prime, and let $K \subset G$ be a $p$-Sylow subgroup.
By Proposition~\ref{prop:resgeomfix}(a,d), the restriction
$\res^G_K(B^{\wedge G}) = C^{\wedge K}$ is the $K$-fold smash power
of $C \cong B^{\wedge G/K}$, which is also bounded below and of
finite type.  In particular, $H_*(C; \bF_p)$ is of finite type, so by
Theorem~\ref{thm:segal-smash-pgroup} the map
$$
\gamma \: (C^{\wedge K})^K \longto (C^{\wedge K})^{hK}
$$
becomes an equivalence after $p$-completion.  

The $K$-spectrum $C^{\wedge K}$ will not generally be
split, so to translate between $p$-adic and $I(K)$-adic
completion we need a replacement for~\cite{MM82}*{Prop.~14}.
Our Proposition~\ref{prop:peqIKeq} will be deduced from the following
result of Ragnarsson, relating the spectrum level $I(K)$-completion of
Greenlees--May~\cite{GM92}*{\S1} to $p$-completion.

\begin{theorem}[\cite{Rag11}*{Thm.~C}]
\label{thm:Rag}
Let $K$ be a $p$-group and $Y$ a bounded below $K$-spectrum.  Then there
is a natural homotopy cofiber sequence
$$
(EK_+ \wedge Y)^K
	\longto (Y^\wedge_{I(K)})^K
	\longto ((\wEK \wedge Y)^K)^\wedge_p \,.
$$
\end{theorem}

\begin{proposition} \label{prop:peqIKeq}
Let $K$ be a $p$-group and $Y$ a bounded below $K$-spectrum with
$\pi_*(Y)$ of finite type.  Suppose that $\gamma \: Y^K \to Y^{hK}$
becomes an equivalence after $p$-completion.  Then
$$
\xi^* \: (Y^\wedge_{I(K)})^K \overset{\simeq}\longto Y^{hK}
$$
is an equivalence.
\end{proposition}

\begin{proof}
For brevity, let $Z = F(EK_+, Y)$.  Note that the $K$-Tate construction
$$
Y^{tK} = (\wEK \wedge Z)^K = (\wEK \wedge F(EK_+, Y))^K \,,
$$
is already $p$-complete, since $K$ is a $p$-group and $\pi_*(Y)$
is bounded below and of finite type, cf.~\cite{GM95}.  Using
Theorem~\ref{thm:Rag} we have vertical maps
$$
\xymatrix{
(EK_+ \wedge Y)^K \ar[r] \ar@{=}[d]
	& Y^K \ar[r] \ar[d]^-{\iota}
	& (\wEK \wedge Y)^K \ar[d]^-{\tilde\iota} \\
(EK_+ \wedge Y)^K \ar[r] \ar[d]_-{\simeq}
	& (Y^\wedge_{I(K)})^K \ar[r] \ar[d]^-{\xi^*}
	& ((\wEK \wedge Y)^K)^\wedge_p \ar[d]^-{\tilde\xi^*} \\
(EK_+ \wedge Z)^K \ar[r]
	& Z^K \ar[r]
	& (\wEK \wedge Z)^K
}
$$
of horizontal homotopy cofiber sequences.  The left hand vertical map
is always an equivalence, and the middle vertical composite~$\gamma =
\xi^* \iota$ becomes an equivalence after $p$-completion by assumption.
Hence also the right hand composite $\tilde\gamma = \tilde\xi^* \tilde\iota$
becomes an equivalence after $p$-completion.  But $\tilde\gamma^\wedge_p
= \tilde\xi^*$, since $(\wEK \wedge Z)^K$ is $p$-complete, so
$\tilde\xi^*$ is an equivalence.  It follows that $\xi^*$ is an equivalence,
as claimed.
\end{proof}

To verify the equivariant bounded below hypothesis for $Y = C^{\wedge K}$,
and a finite type hypothesis needed for~\cite{GM92}*{Thm.~1.6(ii)},
we can use the following variant of the folklore result proved
in~\cite{Rag11}*{Prop.~3.1}.

\begin{lemma} \label{lem:equivbddbelow}
Let $K$ be a finite group and $Y$ a $K$-spectrum.  Let $H$ range over all
subgroups of~$K$.  Then every fixed point spectrum $Y^H$ is bounded below
(and of finite type) if and only if every geometric fixed point spectrum
$\Phi^H(Y)$ is bounded below (and of finite type).
\end{lemma}

\begin{proof}
It suffices to prove this with ``bounded below'' replaced by
``connective''.  Let~$\cP$ be the family of proper subgroups of~$K$. By
induction on~$K$ we may assume that $Y^H$ and $\Phi^H(Y)$ are connective
(and of finite type) for all $H \in \cP$.  In the homotopy cofiber
sequence
$$
(E\cP_+ \wedge Y)^K \longto Y^K \longto \Phi^K(Y)
$$
the left hand term is built from non-negative suspensions of $(K/H_+
\wedge Y)^K \simeq Y^H$, where $H \in \cP$ (and the suspension degrees
increase to infinity), hence is connective (and of finite type).
Thus $Y^K$ is connective (and of finite type) if and only if $\Phi^K(Y)$
is connective (and of finite type).
\end{proof}

\begin{proof}[Proof of Theorem~\ref{thm:segal-smash-finite}]
We keep the notation from the beginning of this section.
By Proposition~\ref{prop:resgeomfix} each geometric fixed
point spectrum $\Phi^H(C^{\wedge K})$ is bounded below and
of finite type,
so by Lemma~\ref{lem:equivbddbelow} the $K$-spectrum $Y =
C^{\wedge K}$ is bounded below and of finite type.
Theorem~\ref{thm:segal-smash-pgroup} for~$K$ and~$C$ and
Proposition~\ref{prop:peqIKeq} then prove that
$$
\xi^* \: (Y^\wedge_{I(K)})^K \overset{\simeq}\longto Y^{hK}
$$
is an equivalence.  Moreover, by~\cite{GM92}*{Thm.~1.6(ii)},
$$
\pi_*((Y^\wedge_{I(K)})^K) \cong \pi_*(Y^K)^\wedge_{I(K)}
$$
is given algebraically by $I(K)$-adic completion.  Hence
Theorem~\ref{thm:segal-smash-finite} holds for the $K$-spectrum $Y$ given
by the restriction of the $G$-spectrum $X = B^{\wedge G}$.  The algebraic
part of this statement is the completion conjecture for $Y$, in the
terminology of May--McClure~\cite{MM82}*{p.~217}.  Since this applies
for all Sylow subgroups~$K$ of~$G$, the completion conjecture also holds
for the $G$-spectrum $X$ by~\cite{MM82}*{Thm.~13}, so that
$$
\pi_*(\gamma)^\wedge_{I(G)} \: \pi_*(X^G)^\wedge_{I(G)}
	\overset{\cong}\longto \pi_*(X^{hG})
$$
is an isomorphism.  Since the $G$-spectrum $X$ is also bounded below
and of finite type we have
$$
\pi_*((X^\wedge_{I(G)})^G) \cong \pi_*(X^G)^\wedge_{I(G)} \,,
$$
by a second appeal to~\cite{GM92}*{Thm.~1.6(ii)}.  Hence
$$
\xi^* \: (X^\wedge_{I(G)})^G \overset{\simeq}\longto X^{hG}
$$
is an equivalence, which proves Theorem~\ref{thm:segal-smash-finite}
for~$G$ and~$B$.
\end{proof}

\end{document}